\documentclass[preprint,12pt]{elsarticle}
\usepackage{amssymb}
\usepackage{amsthm}
\usepackage{amsmath}
\usepackage{epic}
\newtheorem{thm}{Theorem}[section]
\newtheorem{cor}[thm]{Corollary}
\newtheorem{lem}[thm]{Lemma}

\newtheorem{definition}[thm]{Definition}
\newtheorem{con}[thm]{Conjecture}
\journal{arXiv}

\begin{document}

\begin{frontmatter}

\title{Some spectral properties of uniform hypergraphs}
\author[label1,label2]{Jiang Zhou}\ead{zhoujiang04113112@163.com}
\author[label3]{Lizhu Sun}
\author[label1]{Wenzhe Wang}
\author[label1]{Changjiang Bu}\ead{buchangjiang@hrbeu.edu.cn}

\address[label1]{College of Science, Harbin Engineering University, Harbin 150001, PR China}
\address[label2]{College of Computer Science and Technology, Harbin Engineering University, Harbin 150001, PR China}
\address[label3]{School of Science, Harbin Institute of Technology, Harbin 150001, PR China}

\begin{abstract}
For a $k$-uniform hypergraph $H$, we obtain some trace formulas for the Laplacian tensor of $H$, which imply that $\sum_{i=1}^nd_i^s$ ($s=1,\ldots,k$) is determined by the Laplacian spectrum of $H$, where $d_1,\ldots,d_n$ is the degree sequence of $H$. Using trace formulas for the Laplacian tensor, we obtain expressions for some coefficients of the Laplacian polynomial of a regular hypergraph. We give some spectral characterizations of odd-bipartite hypergraphs, and give a partial answer to a question posed by Shao et al \cite{ShaoShanWu}. We also give some spectral properties of power hypergraphs, and show that a conjecture posed by Hu et al \cite{HuQiShao} holds under certain conditons.
\end{abstract}


\end{frontmatter}

\section{Introduction}
Recently, the research on spectral theory of hypergraphs has attracted extensive attention [1,5-8,11,13,14,16-18]. We first introduce some necessary concepts and notations. For a positive integer $n$, let $[n]=\{1,\ldots,n\}$. An order $k$ dimension $n$ tensor $\mathcal{A}=(a_{i_1\cdots i_k})\in\mathbb{C}^{n\times\cdots\times n}$ is a multidimensional array with $n^k$ entries, where $i_j\in[n]$, $j=1,\ldots,k$. We sometimes write $a_{i_1\cdots i_k}$ as $a_{i_1\alpha}$, where $\alpha=i_2\cdots i_k$. When $k=1$, $\mathcal{A}$ is a column vector of dimension $n$. When $k=2$, $\mathcal{A}$ is an $n\times n$ matrix. The \textit{unit tensor} of order $k\geqslant2$ and dimension $n$ is a diagonal tensor $\mathcal{I}_n=(\delta_{i_1i_2\cdots i_k})$ such that $\delta_{i_1i_2\cdots i_k}=1$ if $i_1=i_2=\cdots=i_k$, and $\delta_{i_1i_2\cdots i_k}=0$ otherwise. In \cite{Shao-product}, Shao defined the following product of tensors, which is a generalization of the matrix multiplication.
\begin{definition}\label{definition1.1}\textup{\cite{Shao-product}}
Let $\mathcal{A}$ and $\mathcal{B}$ be order $m\geqslant2$ and order $k\geqslant1$, dimension $n$ tensors, respectively. The product $\mathcal{A}\mathcal{B}$ is the following tensor $\mathcal{C}$ of order $(m-1)(k-1)+1$ and dimension $n$ with entries
\begin{eqnarray*}
c_{i\alpha_1\ldots \alpha_{m-1}}=\sum_{i_2,\ldots,i_m\in[n]}a_{ii_2\ldots i_m}b_{i_2\alpha_1}\cdots b_{i_m\alpha_{m-1}}~(i\in[n],~\alpha_1,\ldots,\alpha_{m-1}\in[n]^{k-1}).
\end{eqnarray*}
\end{definition}

Let $\mathcal{A}$ be an order $k\geqslant2$ dimension $n$ tensor, and let $x=(x_1,\ldots,x_n)^\top$. From Definition \ref{definition1.1}, the product $\mathcal{A}x$ is a vector in $\mathbb{C}^n$ whose $i$-th component is (see Example 1.1 in \cite{Shao-product})
\begin{eqnarray*}
(\mathcal{A}x)_i=\sum_{i_2,\ldots,i_k\in[n]}a_{ii_2\cdots i_k}x_{i_2}\cdots x_{i_k}.
\end{eqnarray*}
The concept of tensor eigenvalues was posed in \cite{Lim,Qi05}. If there exists a nonzero vector $x\in\mathbb{C}^n$ such that $\mathcal{A}x=\lambda x^{[k-1]}$, then $\lambda$ is called an \textit{eigenvalue} of $\mathcal{A}$, $x$ is an \textit{eigenvector} of $\lambda$, where $x^{[k-1]}=(x_1^{k-1},\ldots,x_n^{k-1})^\top$. The \textit{determinant} of $\mathcal{A}$, denoted by $\det(\mathcal{A})$, is the resultant of the system of polynomials $f_i(x_1,\ldots,x_n)=(\mathcal{A}x)_i$ ($i=1,\ldots,n$). The \textit{characteristic polynomial} of $\mathcal{A}$ is defined as $\phi_{\mathcal{A}}(\lambda)=\det(\lambda\mathcal{I}_n-\mathcal{A})$, where $\mathcal{I}_n$ is the unit tensor of order $k$ and dimension $n$. It is known that eigenvalues of $\mathcal{A}$ are exactly roots of $\phi_{\mathcal{A}}(\lambda)$ \cite{Qi05}. The multiset of roots of $\phi_{\mathcal{A}}(\lambda)$ (counting multiplicities) is the \textit{spectrum} of $\mathcal{A}$, denoted by $\sigma(\mathcal{A})$. The maximal modulus of eigenvalues of $\mathcal{A}$ is called the \textit{spectral radius} of $\mathcal{A}$, denoted by $\rho(\mathcal{A})$. More details on eigenvalues and characteristic polynomials of tensors can be found in \cite{Hu2013,Qi05}.

A hypergraph $H$ is called $k$-\textit{uniform} if each edge of $H$ contains exactly $k$ distinct vertices. Let $V(H)$ and $E(H)$ denote the vertex set and the edge set of $H$, respectively. In \cite{Qi-Laplaican}, Qi defined the Laplacian and the signless Laplacian tensor of a uniform hypergraph as follows.
\begin{definition}\textup{\cite{HuQiShao,Qi-Laplaican}}
The \textit{adjacency tensor} of a $k$-uniform hypergraph $H$, denoted by $\mathcal{A}_H$, is an order $k$ dimension $|V(H)|$ tensor with entries
\begin{eqnarray*}
a_{i_1i_2\cdots i_k}=\begin{cases}\frac{1}{(k-1)!}~~~~~~~\mbox{if}~i_1i_2\cdots i_k\in E(H),\\
0~~~~~~~~~~~~~\mbox{otherwise}.\end{cases}
\end{eqnarray*}
Let $\mathcal{D}_H$ be an order $k$ dimension $|V(H)|$ diagonal tensor whose diagonal entries are vertex degrees of $H$. The tensors $\mathcal{L}_H=\mathcal{D}_H-\mathcal{A}_H$ and $\mathcal{Q}_H=\mathcal{D}_H+\mathcal{A}_H$ are the \textit{Laplacian tensor} and the \textit{signless Laplacian tensor} of $H$, respectively. Eigenvalues of $\mathcal{A}_H$, $\mathcal{L}_H$ and $\mathcal{Q}_H$ are called eigenvalues, Laplacian eigenvalues and signless Laplacian eigenvalues of $H$, respectively. Characteristic polynomials of $\mathcal{L}_H$ and $\mathcal{Q}_H$ are called Laplacian polynomial and signless Laplacian polynomial of $H$, respectively.
\end{definition}
This paper is organized as follows. In Section 2, we give some trace formulas for the Laplacian tensor of a uniform hypergraph, and obtain expressions for some coefficients of the Laplacian polynomial of a regular hypergraph. In Section 3, we give some spectral characterizations of odd-bipartite hypergraphs. In Section 4, we give some spectral properties of power hypergraphs.
\section{Laplacian spectra and degree sequence of hypergraphs}
Traces of tensors are very useful in the study of spectral theory of tensors. The \textit{d-th order trace} of an order $k\geqslant2$ dimension $n$ tensor $\mathcal{T}=(t_{i_1\cdots i_k})$ is defined as \cite{Cooper12,Hu2013,Morozov}
\begin{eqnarray*}
Tr_d(\mathcal{T})=(k-1)^{n-1}\sum_{d_1+\cdots+d_n=d}\prod_{i=1}^n\frac{1}{(d_i(k-1))!}\left(\sum_{y\in[n]^{k-1}}t_{iy}\frac{\partial}{\partial a_{iy}}\right)^{d_i}\mbox{tr}(A^{d(k-1)}),
\end{eqnarray*}
where $A=(a_{ij})$ is an $n\times n$ auxiliary matrix, and $\frac{\partial}{\partial a_{iy}}=\frac{\partial}{\partial a_{ii_2}}\cdots\frac{\partial}{\partial a_{ii_k}}$ if $y=i_2\cdots i_k$. The codegree $d$ coefficient of the characteristic polynomial of $\mathcal{T}$ can be expressed in terms of $Tr_1(\mathcal{T}),\ldots,Tr_d(\mathcal{T})$ (see [4, Theorem 6.3]). It is also known that $Tr_t(\mathcal{T})=\sum_{\lambda\in\sigma(\mathcal{T})}\lambda^t$ for any $t\in[n(k-1)^{n-1}]$ (see [4, Theorem 6.10]). Hence $Tr_d(\mathcal{T})$ is an important invariant in the spectral theory of tensors.

Shao et al \cite{ShaoQiHu} give a graph theoretical formula for $Tr_d(\mathcal{T})$. In order to describe this formula, we introduce some notations in \cite{ShaoQiHu}. For an integer $d>0$, we define
\begin{eqnarray*}
\mathcal{F}_d=\{(i_1\alpha_1,\ldots,i_d\alpha_d)|1\leqslant i_1\leqslant\cdots\leqslant i_d\leqslant n;~\alpha_1,\ldots,\alpha_d\in[n]^{k-1}\}.
\end{eqnarray*}

For $F=(i_1\alpha_1,\ldots,i_d\alpha_d)\in\mathcal{F}_d$ and an order $k\geqslant2$ dimension $n$ tensor $\mathcal{T}=(t_{i_1\cdots i_k})$, we write $\pi_F(\mathcal{T})=\prod_{j=1}^dt_{i_j\alpha_j}$. Let $p_i(F)$ be the total number of times that the index $i$ appears in $F$. If $p_i(F)$ is a multiple of $k$ for any $i\in[n]$, then $F$ is called \textit{k-valent}.

\begin{definition}\label{definition2.1}\textup{\cite{ShaoQiHu}}
Let $F=(i_1\alpha_1,\ldots,i_d\alpha_d)\in\mathcal{F}_d$, where $i_j\alpha_j\in[n]^k$, $j=1,\ldots,d$. Then\\
(1) Let $E(F)=\bigcup_{j=1}^dE_j(F)$, where $E_j(F)$ is the arc multi-set
\begin{eqnarray*}
E_j(F)=\{(i_j,v_1),\ldots,(i_j,v_{k-1})\}
\end{eqnarray*}
if $\alpha_j=v_1\cdots v_{k-1}$.\\
(2) Let $b(F)$ be the product of the factorials of the multiplicities of all the arcs of $E(F)$.\\
(3) Let $c(F)$ be the product of the factorials of the outdegrees of all the vertices in the arc multi-set $E(F)$.\\
(4) Let $W(F)$ be the set of all closed walks $W$ with the arc multi-set $E(F)$.
\end{definition}
Shao et al give a graph theoretical formula for $Tr_d(\mathcal{T})$ as follows (see equation (3.5) in \cite{ShaoQiHu}).
\begin{lem}\textup{\cite{ShaoQiHu}}\label{lem2.2}
Let $\mathcal{T}=(T_{i_1\cdots i_k})$ be an order $k\geqslant2$ dimension $n$ tensor. Then
\begin{eqnarray*}
Tr_d(\mathcal{T})=(k-1)^{n-1}\sum_{F\in\mathcal{F}_d^\prime}\frac{b(F)}{c(F)}\pi_F(\mathcal{T})|W(F)|,
\end{eqnarray*}
where $\mathcal{F}_d^\prime=\{F|F\in\mathcal{F}_d,~F~\mbox{is k-valent}\}$.
\end{lem}
For a $k$-uniform hypergraph $H$, Cooper and Dutle \cite{Cooper12} proved that $Tr_d(\mathcal{A}_H)=0$ for $d\in[k-1]$. We give some trace formulas for the Laplacian (signless Laplacian) tensor of uniform hypergraphs as follows.
\begin{thm}\label{thm2.3}
Let $H$ be a $k$-uniform hypergraph with degree sequence $d_1,\ldots,d_n$. Then
\begin{eqnarray*}
Tr_t(\mathcal{L}_H)&=&Tr_t(\mathcal{Q}_H)=(k-1)^{n-1}\sum_{i=1}^nd_i^t,~t=1,\ldots,k-1,\\
Tr_k(\mathcal{L}_H)&=&(-1)^kk^{k-1}(k-1)^{n-k}|E(H)|+(k-1)^{n-1}\sum_{i=1}^nd_i^k,\\
Tr_k(\mathcal{Q}_H)&=&k^{k-1}(k-1)^{n-k}|E(H)|+(k-1)^{n-1}\sum_{i=1}^nd_i^k.
\end{eqnarray*}
\end{thm}
\begin{proof}
By Lemma \ref{lem2.2}, we have
\begin{eqnarray}
Tr_t(\mathcal{L}_H)=(k-1)^{n-1}\sum_{F\in\mathcal{F}_t^\prime}\frac{b(F)}{c(F)}\pi_F(\mathcal{L}_H)|W(F)|,
\end{eqnarray}
where $\mathcal{F}_t^\prime=\{F|F\in\mathcal{F}_t,~F~\mbox{is k-valent}\}$. For $F=(i_1\alpha_1,\ldots,i_t\alpha_t)\in\mathcal{F}_t$, if $\pi_F(\mathcal{L}_H)=\prod_{j=1}^t(\mathcal{L}_H)_{i_j\alpha_j}\neq0$, then $i_j\alpha_j=i_ji_j\cdots i_j\in[n]^k$ or $i_j\alpha_j\in E(H)$ for any $1\leqslant j\leqslant t$.

Let $F\in\mathcal{F}_t$ satisfies $\pi_F(\mathcal{L}_H)\neq0$. If $t<k$, then $F$ is k-valent if and only if $F=(i_1i_1\cdots i_1,\ldots,i_ti_t\cdots i_t)$. In this case, $|W(F)|\neq0$ if and only if $i_1=\cdots=i_t$. Let $F_i=(ii\cdots i,\ldots,ii\cdots i)\in\mathcal{F}_t^\prime$ ($t<k$). From Eq. (1) and Definition \ref{definition2.1}, we have
\begin{eqnarray*}
Tr_t(\mathcal{L}_H)&=&(k-1)^{n-1}\sum_{i=1}^n\frac{b(F_i)}{c(F_i)}\pi_{F_i}(\mathcal{L}_H)|W(F_i)|\\
&=&(k-1)^{n-1}\sum_{i=1}^n\frac{(t(k-1))!}{(t(k-1))!}d_i^t=(k-1)^{n-1}\sum_{i=1}^nd_i^t.
\end{eqnarray*}
Similar with the above procedure, we can also get $Tr_t(\mathcal{Q}_H)=(k-1)^{n-1}\sum_{i=1}^nd_i^t$, $t=1,\ldots,k-1$.

Let $F\in\mathcal{F}_k$ satisfies $\pi_F(\mathcal{L}_H)\neq0$. Then $F$ is k-valent and $|W(F)|\neq0$ if and only if $F=(ii\cdots i,\ldots,ii\cdots i)$ or $F=(i_1\alpha_1,\ldots,i_k\alpha_k)$, where $i_1\alpha_1,\ldots,i_k\alpha_k$ correspond to the same edge $i_1i_2\cdots i_k\in E(H)$. Let $F_i=(ii\cdots i,\ldots,ii\cdots i)\in\mathcal{F}_k^\prime$. From Eq. (1) and Definition \ref{definition2.1}, we have
\begin{eqnarray*}
Tr_k(\mathcal{L}_H)&=&(-1)^kTr_k(\mathcal{A}_H)+(k-1)^{n-1}\sum_{i=1}^n\frac{b(F_i)}{c(F_i)}\pi_{F_i}(\mathcal{L}_H)|W(F_i)|\\
&=&(-1)^kTr_k(\mathcal{A}_H)+(k-1)^{n-1}\sum_{i=1}^n\frac{(k(k-1))!}{(k(k-1))!}d_i^k\\
&=&(-1)^kTr_k(\mathcal{A}_H)+(k-1)^{n-1}\sum_{i=1}^nd_i^k.
\end{eqnarray*}
From the proof of [1, Theorem 3.15], we have $Tr_k(\mathcal{A}_H)=k^{k-1}(k-1)^{n-k}|E(H)|$. Hence
\begin{eqnarray*}
Tr_k(\mathcal{L}_H)=(-1)^kk^{k-1}(k-1)^{n-k}|E(H)|+(k-1)^{n-1}\sum_{i=1}^nd_i^k.
\end{eqnarray*}
Similar with the above procedure, we can also get
\begin{eqnarray*}
Tr_k(\mathcal{Q}_H)=k^{k-1}(k-1)^{n-k}|E(H)|+(k-1)^{n-1}\sum_{i=1}^nd_i^k.
\end{eqnarray*}
\end{proof}

\noindent
\textbf{Remark.} Note that traces of a tensor are determined by its spectrum [3, Theorem 6.3]. For a $k$-uniform hypergraph $H$, by Theorem \ref{thm2.3}, we know that $\sum_{i=1}^nd_i^s$ ($s=1,\ldots,k$) is determined by the Laplacian (signless Laplacian) spectrum of $H$, where $d_1,\ldots,d_n$ is the degree sequence of $H$.

\vspace{3mm}
Let $p_t(\mathcal{M})$ denote the codegree $t$ coefficient of the characteristic polynomial of a tensor $\mathcal{M}$.
\begin{lem}\label{lem2.4}
Let $\mathcal{M}$ be an order $k\geqslant2$ dimension $n$ tensor. Then
\begin{eqnarray*}
t!p_t(\mathcal{M})=\det\begin{pmatrix}-Tr_t&Tr_1&Tr_2&\cdots&Tr_{t-1}\\
-Tr_{t-1}&t-1&Tr_1&\cdots&Tr_{t-2}\\
-Tr_{t-2}&0&t-2&\ddots&\vdots\\
\vdots&\vdots&\ddots&\ddots&Tr_1\\
-Tr_1&0&\cdots&0&1\end{pmatrix},
\end{eqnarray*}
where $Tr_t=Tr_t(\mathcal{M})$, $t\in[n(k-1)^{n-1}]$.
\end{lem}
\begin{proof}
From [4, Theorem 6.10], we have
\begin{eqnarray*}
\begin{pmatrix}t&Tr_1&Tr_2&\cdots&Tr_{t-1}\\
0&t-1&Tr_1&\cdots&Tr_{t-2}\\
0&0&t-2&\ddots&\vdots\\
\vdots&\vdots&\ddots&\ddots&Tr_1\\
0&0&\cdots&0&1\end{pmatrix}
\begin{pmatrix}p_t(\mathcal{M})\\p_{t-1}(\mathcal{M})\\ \vdots\\p_2(\mathcal{M})\\p_1(\mathcal{M})\end{pmatrix}=
\begin{pmatrix}-Tr_t\\-Tr_{t-1}\\ \vdots\\-Tr_2\\-Tr_1\end{pmatrix}.
\end{eqnarray*}
We can obtain the expression of $t!p_t(\mathcal{M})$ from Cramer's rule.
\end{proof}
A uniform hypergraph $H$ is called \textit{d-regular} if each vertex of $H$ has degree $d$. The following are some coefficients of the Laplacian (signless Laplacian) polynomial of regular hypergraphs.
\begin{thm}
Let $H$ be a $d$-regular $k$-uniform hypergraph with $n$ vertices. Then
\begin{eqnarray*}
p_t(\mathcal{L}_H)&=&p_t(\mathcal{Q}_H)=(-1)^td^t\binom{n(k-1)^{n-1}}{t},~t=1,\ldots,k-1,\\
p_k(\mathcal{L}_H)&=&(-1)^{k+1}k^{k-3}(k-1)^{n-k}nd+(-1)^kd^k\binom{n(k-1)^{n-1}}{k},\\
p_k(\mathcal{Q}_H)&=&-k^{k-3}(k-1)^{n-k}nd+(-1)^kd^k\binom{n(k-1)^{n-1}}{k}.
\end{eqnarray*}
\end{thm}
\begin{proof}
By Lemma \ref{lem2.4}, we have
\begin{eqnarray}
t!p_t(\mathcal{L}_H)=\det\begin{pmatrix}-Tr_t&Tr_1&Tr_2&\cdots&Tr_{t-1}\\
-Tr_{t-1}&t-1&Tr_1&\cdots&Tr_{t-2}\\
-Tr_{t-2}&0&t-2&\ddots&\vdots\\
\vdots&\vdots&\ddots&\ddots&Tr_1\\
-Tr_1&0&\cdots&0&1\end{pmatrix},
\end{eqnarray}
where $Tr_t=Tr_t(\mathcal{L}_H)$. Since $H$ is $d$-regular, by Theorem \ref{thm2.3}, we have $Tr_i=dTr_{i-1}=n(k-1)^{n-1}d^i$, $i=2,\ldots,k-1$. If $t<k$, then by Eq. (2), we have
\begin{eqnarray*}
&&t!p_t(\mathcal{L}_H)=\det\begin{pmatrix}0&Tr_1&Tr_2&\cdots&Tr_{t-1}\\
0&t-1&Tr_1&\cdots&Tr_{t-2}\\
\vdots&0&t-2&\ddots&\vdots\\
0&\vdots&\ddots&\ddots&Tr_1\\
d-Tr_1&0&\cdots&0&1\end{pmatrix}\\
&=&\det\begin{pmatrix}0&Tr_1-(t-1)d&0&\cdots&0\\
0&t-1&\ddots&\ddots&\vdots\\
\vdots&0&\ddots&Tr_1-2d&0\\
0&\vdots&\ddots&2&Tr_1\\
d-Tr_1&0&\cdots&0&1\end{pmatrix}\\
&=&(-1)^t\prod_{i=0}^{t-1}(Tr_1-id).
\end{eqnarray*}
Since $Tr_1=n(k-1)^{n-1}d$, we have
\begin{eqnarray*}
p_t(\mathcal{L}_H)&=&(-1)^t\frac{\prod_{i=0}^{t-1}(n(k-1)^{n-1}d-id)}{t!}=(-1)^td^t\frac{\prod_{i=0}^{t-1}(n(k-1)^{n-1}-i)}{t!}\\
&=&(-1)^td^t\binom{n(k-1)^{n-1}}{t}.
\end{eqnarray*}
Similar with the above procedure, we can also get
\begin{eqnarray*}
p_t(\mathcal{Q}_H)=(-1)^td^t\binom{n(k-1)^{n-1}}{t},~t=1,\ldots,k-1.
\end{eqnarray*}

Since $H$ is $d$-regular, by Theorem \ref{thm2.3}, we have $Tr_k=(-1)^kk^{k-2}(k-1)^{n-k}nd+dTr_{k-1}$ and $Tr_i=dTr_{i-1}=n(k-1)^{n-1}d^i$, $i=2,\ldots,k-1$. From Eq. (2), we have
\begin{eqnarray*}
&~&k!p_k(\mathcal{L}_H)=\det\begin{pmatrix}(-1)^{k+1}k^{k-2}(k-1)^{n-k}nd&Tr_1&Tr_2&\cdots&Tr_{k-1}\\
0&k-1&Tr_1&\cdots&Tr_{k-2}\\
\vdots&0&k-2&\ddots&\vdots\\
0&\vdots&\ddots&\ddots&Tr_1\\
d-Tr_1&0&\cdots&0&1\end{pmatrix}\\
&=&\det\begin{pmatrix}(-1)^{k+1}k^{k-2}(k-1)^{n-k}nd&Tr_1-(k-1)d&0&\cdots&0\\
0&k-1&\ddots&\ddots&\vdots\\
\vdots&0&\ddots&Tr_1-2d&0\\
0&\vdots&\ddots&2&Tr_1\\
d-Tr_1&0&\cdots&0&1\end{pmatrix}\\
&=&(-1)^{k+1}k^{k-2}(k-1)^{n-k}(k-1)!nd+(-1)^kd^k\prod_{i=0}^{k-1}(n(k-1)^{n-1}-i).
\end{eqnarray*}
\begin{eqnarray*}
p_k(\mathcal{L}_H)=(-1)^{k+1}k^{k-3}(k-1)^{n-k}nd+(-1)^kd^k\binom{n(k-1)^{n-1}}{k}.
\end{eqnarray*}
Similar with the above procedure, we can also get
\begin{eqnarray*}
p_k(\mathcal{Q}_H)&=&-k^{k-3}(k-1)^{n-k}nd+(-1)^kd^k\binom{n(k-1)^{n-1}}{k}.
\end{eqnarray*}
\end{proof}
\section{Eigenvalues and odd-bipartite hypergraphs}
A $k$-uniform hypergraph $H$ is called \textit{odd-bipartite}, if there exists a proper subset $V_1$ of $V(H)$ such that each edge of $H$ contains exactly odd number of vertices in $V_1$ \cite{HuQi-DAM,ShaoShanWu}. Spectral characterizations of odd-bipartite hypergraphs will be investigated in this section. We first give some auxiliary lemmas. The following lemma can be obtained from equation (2.1) in \cite{Shao-product}.
\begin{lem}\label{lem3.1}
Let $\mathcal{A}=(a_{i_1\cdots i_k})$ be an order $k\geqslant2$ dimension $n$ tensor, and let $P=(p_{ij}),Q=(q_{ij})$ be $n\times n$ matrices. Then
\begin{eqnarray*}
(P\mathcal{A}Q)_{i_1\cdots i_k}=\sum_{j_1,\ldots,j_k\in[n]}a_{j_1\cdots j_k}p_{i_1j_1}q_{j_2i_2}\cdots q_{j_ki_k}.
\end{eqnarray*}
\end{lem}
\begin{lem}\textup{\cite{HuQi-DAM}}\label{lem3.2}
Let $H$ be a connected k-uniform hypergraph. A nonzero vector $x$ is an eigenvector of $\mathcal{Q}_H$ corresponds to the zero eigenvalue if and only if there exist nonzero $\gamma\in\mathbb{C}$ and integers $\alpha_i$ such that $x_i=\gamma\exp(\frac{2\alpha_i\pi}{k}\sqrt{-1})$ for each $i\in V(H)$, and
\begin{eqnarray*}
\sum_{j\in e}\alpha_j=\sigma_ek+\frac{k}{2}
\end{eqnarray*}
for some integer $\sigma_e$ associated with each $e\in E(H)$.
\end{lem}
Weakly irreducible tensors are defined in \cite{Friedland}. It is known that a $k$-uniform hypergraph $H$ is connected if and only if $\mathcal{A}_H$ is weakly irreducible \cite{Pearson}.
\begin{lem}\label{lem3.3}\textup{\cite{ShaoShanWu,Yang}}
Let $\mathcal{A}$ be an order $k$ dimension $n$ nonnegative weakly irreducible tensor. If $\rho(\mathcal{A})\exp(\theta\sqrt{-1})$ is an eigenvalue of $\mathcal{A}$, then there exists a diagonal matrix $U$ with unit diagonal entries such that
\begin{eqnarray*}
\mathcal{A}=\exp(-\theta\sqrt{-1})U^{-(k-1)}\mathcal{A}U.
\end{eqnarray*}
\end{lem}
For a tensor $\mathcal{T}$, let $H\sigma(\mathcal{T})=\{\lambda|\lambda\in\sigma(\mathcal{T}),~\lambda~\mbox{has a real eigenvector}\}$. For a connected $k$-uniform hypergraph $G$, Shao et al \cite{ShaoShanWu} proved that
\begin{eqnarray*}
H\sigma(\mathcal{L}_G)=H\sigma(\mathcal{Q}_G)\Longrightarrow\sigma(\mathcal{L}_G)=\sigma(\mathcal{Q}_G).
\end{eqnarray*}
Shao et al wish to know whether the reverse implication is true. We show that the reverse is true when $k$ is not divisible by $4$.
\begin{thm}\label{thm3.4}
Let $G$ be a connected $k$-uniform hypergraph, and $k$ is not divisible by $4$. Then the following are equivalent:\\
(1) $k$ is even and $H$ is odd-bipartite.\\
(2) $H\sigma(\mathcal{L}_G)=H\sigma(\mathcal{Q}_G)$.\\
(3) $\sigma(\mathcal{L}_G)=\sigma(\mathcal{Q}_G)$.\\
(4) $0$ is a signless Laplacian eigenvalue of $G$.
\end{thm}
\begin{proof}
From [17, Theorem 2.2], we have $(1)\Rightarrow(2)\Rightarrow(3)$. Since $0$ is always an eigenvalue of $\mathcal{L}_G$ (see \cite{Qi-Laplaican}), we have $(3)\Rightarrow(4)$. Next we prove that $(4)\Rightarrow(1)$.

If $0$ is an eigenvalue of $\mathcal{Q}_G$, then by Lemma \ref{lem3.2}, there exists a vertex labeling $f:V(G)\rightarrow[k]$ such that
\begin{eqnarray*}
\sum_{i\in e}f(i)\equiv\frac{k}{2}~(\mbox{mod}~k)
\end{eqnarray*}
for each $e\in E(G)$. Hence $k$ is even. Since $k$ is not divisible by $4$, we know that $\frac{k}{2}$ is odd. So $\sum_{i\in e}f(i)$ is odd for each $e\in E(G)$. Let $V_1=\{u|u\in V(G),~f(u)~\mbox{is odd}\}$. For any $e\in E(G)$, since $\sum_{i\in e}f(i)$ is odd, $e$ contains exactly odd number of vertices in $V_1$. Hence $G$ is odd-bipartite.
\end{proof}
When $k=2$, Theorem \ref{thm3.4} becomes a classic result in spectral graph theory, i.e., a connected graph $G$ is bipartite if and only if $0$ is a signless Laplacian eigenvalue of $G$. It is also well known that a connected graph $G$ is bipartite if and only if $-\rho(\mathcal{A}_G)$ is an eigenvalue of $G$. We generalize this result as follows.
\begin{thm}\label{thm3.5}
Let $H$ be a connected $k$-uniform hypergraph, and $k$ is not divisible by $4$. Then the following are equivalent:\\
(1) $k$ is even and $H$ is odd-bipartite.\\
(2) $-\rho(\mathcal{A}_H)$ is an eigenvalue of $H$.
\end{thm}
\begin{proof}
From [17, Theorem 2.3], we have $(1)\Rightarrow(2)$. If (2) holds, then by Lemma \ref{lem3.3}, there exists a diagonal matrix $U$ with unit diagonal entries such that $\mathcal{A}_H=-U^{-(k-1)}\mathcal{A}_HU$. By Lemma \ref{lem3.1}, we have
\begin{eqnarray*}
a_{i_1i_2\cdots i_k}=-a_{i_1i_2\cdots i_k}u_{i_1}^{-(k-1)}u_{i_2}\cdots u_{i_k},
\end{eqnarray*}
where $u_{i_j}$ is the diagonal entry of $U$ corresponds to vertex $i_j$ ($j=1,\ldots,k$). For any edge $e=i_1i_2\cdots i_k\in E(H)$, we get
\begin{eqnarray*}
u_{i_1}^{-(k-1)}u_{i_2}\cdots u_{i_k}=-1,~u_{i_1}u_{i_2}\cdots u_{i_k}=-u_{i_1}^k.
\end{eqnarray*}
Similarly, we have $u_{i_1}u_{i_2}\cdots u_{i_k}=-u_{i_1}^k=\cdots=-u_{i_k}^k$. Since $u_{i_1},\ldots,u_{i_k}$ are unit complex number, there exist integers $\alpha_{i_1},\ldots,\alpha_{i_k}$ and $\theta$ such that $u_{i_j}=\exp(\frac{2\pi\alpha_{i_j}+\theta}{k}\sqrt{-1})$, $j=1,\ldots,k$. Then
\begin{eqnarray*}
u_{i_1}u_{i_2}\cdots u_{i_k}=\exp(\frac{k\theta+2\pi\sum_{j=1}^k\alpha_{i_j}}{k}\sqrt{-1})&=&-u_{i_1}^k=-\exp(\theta\sqrt{-1}),\\
\exp(\frac{2\pi\sum_{j=1}^k\alpha_{i_j}}{k}\sqrt{-1})&=&-1.
\end{eqnarray*}
Hence $\sum_{j=1}^k\alpha_{i_j}\equiv\frac{k}{2}~(\mbox{mod}~k)$, $k$ is even. Since $k$ is not divisible by $4$, $\sum_{j=1}^k\alpha_{i_j}$ is odd for any edge $e=i_1i_2\cdots i_k\in E(H)$. Let $V_1=\{u|u\in V(H),\alpha_u~\mbox{is odd}\}$. For any $e=i_1i_2\cdots i_k\in E(H)$, since $\sum_{j=1}^k\alpha_{i_j}$ is odd, $e$ contains exactly odd number of vertices in $V_1$. Hence $H$ is odd-bipartite.
\end{proof}
Let $H$ be a connected $k$-uniform hypergraph. If $0$ is an eigenvalue of $\mathcal{Q}_H$, then by the proof of Theorem \ref{thm3.4}, we know that there exists a vertex labeling $f:V(H)\rightarrow[k]$ such that $\sum_{i\in e}f(i)\equiv\frac{k}{2}~(\mbox{mod}~k)$ for each $e\in E(H)$. We pose the following conjecture.
\begin{con}
Let $H$ be a connected $k$-uniform hypergraph. Then the following are equivalent:\\
(1) $k$ is even and $H$ is odd-bipartite.\\
(2) $0$ is a signless Laplacian eigenvalue of $H$.\\
(3) $-\rho(\mathcal{A}_H)$ is an eigenvalue of $H$.\\
(4) There exists a vertex labeling $f:V(H)\rightarrow[k]$ such that $\sum_{i\in e}f(i)\equiv\frac{k}{2}~(\mbox{mod}~k)$ for each $e\in E(H)$.
\end{con}
\section{Eigenvalues of power hypergraphs}
A vertex with degree one is called a \textit{core vertex} \cite{HuQiShao}. For a $k$-uniform hypergraph $H$, if $e\in E(H)$ contains core vertices, then we use $H-e$ to denote a $k$-uniform sub-hypergraph of $H$ obtained by deleting the edge $e$ and all core vertices in $e$.
\begin{thm}\label{thm4.1}
Let $H$ be a $k$-uniform hypergraph, and let $e\in E(H)$ be an edge contains at least two core vertices. If $\lambda$ is an eigenvalue of $H-e$, then $\lambda$ is an eigenvalue of $H$.
\end{thm}
\begin{proof}
Suppose that $x$ is an eigenvector of the eigenvalue $\lambda$ of $H-e$. Let $y$ be a column vector of dimension $|V(H)|$ such that $y_u=x_u$ if $u\in V(H-e)$, and $y_u=0$ if $u\in V(H)$ is a core vertex in $e$. Since $\mathcal{A}_{H-e}x=\lambda x^{[k-1]}$, we have $\mathcal{A}_Hy=\lambda y^{[k-1]}$. So $\lambda$ is an eigenvalue of $H$.
\end{proof}
In \cite{HuQiShao}, Hu et al defined power hypergraphs as follows.
\begin{definition}\textup{\cite{HuQiShao}}
Let $G$ be an ordinary graph (i.e. $2$-uniform hypergraph). For any $k\geqslant3$, the $k$th power of $G$, denoted by $G^k$, is a $k$-uniform hypergraph with edge set $E(G^k)=\{e\cup\{i_{e,1},\ldots,i_{e,k-2}\}|e\in E(G)\}$, and vertex set $V(G^k)=V(G)\cup\{i_{e,j}|e\in E(G),j\in[k-2]\}$.
\end{definition}
Some examples of power hypergraphs are given in [7, Fig.1]. From Definition 4.2, we know that each edge of a power hypergraph $G^k$ contains two adjacent vertices in $V(G)$ and $k-2$ core vertices not in $V(G)$.

If $H$ is a connected $k$-uniform hypergraph, then $\mathcal{A}_H$ and $\mathcal{Q}_H$ are both weakly irreducible \cite{Qi-Laplaican}. So we obtain the following lemma from [13, Theorem 2.2].
\begin{lem}\label{lem4.3}
Let $H$ be a connected $k$-uniform hypergraph. If $\lambda$ is an eigenvalue of $\mathcal{A}_H$ ($\mathcal{Q}_H$) with a positive eigenvector, then $\lambda=\rho(\mathcal{A}_H)$ ($\lambda=\rho(\mathcal{Q}_H)$).
\end{lem}

\begin{thm}\label{thm4.4}
If $\lambda\neq0$ is an eigenvalue of a graph $G$, then $\lambda^{\frac{2}{k}}$ is an eigenvalue of $G^k$. Moreover, $\rho(\mathcal{A}_{G^k})=\rho(\mathcal{A}_G)^{\frac{2}{k}}$.
\end{thm}
\begin{proof}
Suppose that $x$ is an eigenvector of the eigenvalue $\lambda\neq0$ of graph $G$. Then $\sum_{j\in N_G(i)}x_j=\lambda x_i$ for any $i\in V(G)$, where $N_G(i)$ is the set of all neighbors of $i$ in $G$. Let $y$ be a column vector of dimension $|V(G^k)|$ such that $y_u=(x_u)^{\frac{2}{k}}$ if $u\in V(G)$, and $y_u=(\lambda^{-1}x_ix_j)^{\frac{1}{k}}$ if $u\in V(G^k)\setminus V(G)$ is a core vertex in the edge contains two adjacent vertices $i,j\in V(G)$. For any $i\in V(G)$, by $\sum_{j\in N_G(i)}x_j=\lambda x_i$, we have
\begin{eqnarray*}
(\mathcal{A}_{G^k}y)_i=\sum_{j\in N_G(i)}(\lambda^{-1}x_ix_j)^{\frac{k-2}{k}}(x_j)^{\frac{2}{k}}=\lambda^{\frac{2}{k}}(x_i)^{\frac{2(k-1)}{k}}=\lambda^{\frac{2}{k}}(y_i)^{k-1}.
\end{eqnarray*}
For any $u\in V(G^k)\setminus V(G)$, we have
\begin{eqnarray*}
(\mathcal{A}_{G^k}y)_u=(\lambda^{-1}x_ix_j)^{\frac{k-3}{k}}(x_i)^{\frac{2}{k}}(x_j)^{\frac{2}{k}}=\lambda^{\frac{2}{k}}(\lambda^{-1}x_ix_j)^{\frac{k-1}{k}}=\lambda^{\frac{2}{k}}(y_u)^{k-1}.
\end{eqnarray*}
Hence $\lambda^{\frac{2}{k}}$ is an eigenvalue of $G^k$ with an eigenvector $y$.

If $G$ is connected and $\lambda=\rho(\mathcal{A}_G)$, then we can choose $x$ as a positive eigenvector of $\rho(\mathcal{A}_G)$. In this case, $y$ is a positive eigenvector of the eigenvalue $\rho(\mathcal{A}_G)^{\frac{2}{k}}$ of $G^k$. Lemma \ref{lem4.3} implies that $\rho(\mathcal{A}_{G^k})=\rho(\mathcal{A}_G)^{\frac{2}{k}}$ when $G$ is connected.

If $G$ has $r\geqslant2$ components $G_1,\ldots,G_r$, then
\begin{eqnarray*}
\rho(\mathcal{A}_{G^k})=\max\{\rho(\mathcal{A}_{G_1^k}),\ldots,\rho(\mathcal{A}_{G_r^k})\}=\max\{\rho(\mathcal{A}_{G_1})^{\frac{2}{k}},\ldots,\rho(\mathcal{A}_{G_r})^{\frac{2}{k}}\}=\rho(\mathcal{A}_G)^{\frac{2}{k}}.
\end{eqnarray*}
\end{proof}
We can obtain the following result from Theorem \ref{thm4.4}.
\begin{cor}
For any nontrivial graph $G$, we have $\lim_{k\rightarrow\infty}\rho(\mathcal{A}_{G^k})=1$. Moreover, $\{\rho(\mathcal{A}_{G^k})\}$ is a strictly decreasing sequence if $\rho(\mathcal{A}_G)>1$.
\end{cor}
The following corollary follows from Theorem \ref{thm4.1} and \ref{thm4.4}.
\begin{cor}
If $\lambda\neq0$ is an eigenvalue of any subgraph of a graph $G$, then $\lambda^{\frac{2}{k}}$ is an eigenvalue of $G^k$ for $k\geqslant4$.
\end{cor}
Let $P_n$ and $S_n$ be the path and the star of order $n$, respectively. The following result was proved by Li et al \cite{LiShaoQi}. Here we give a different proof.
\begin{cor}
Let $T$ be a tree with $n$ vertices. Then
\begin{eqnarray*}
\rho(\mathcal{A}_{P_n^k})\leqslant\rho(\mathcal{A}_{T^k})\leqslant\rho(\mathcal{A}_{S_n^k}),
\end{eqnarray*}
where the left equality holds if and only if $T=P_n$, and the right equality holds if and only if $T=S_n$.
\end{cor}
\begin{proof}
Among all trees with $n$ vertices, $P_n$ is the unique tree with the smallest adjacency spectral radius, and $S_n$ is the unique tree with the largest adjacency spectral radius \cite{Cvetkovic}. By Theorem \ref{thm4.4}, we have \begin{eqnarray*}
\rho(\mathcal{A}_{P_n^k})\leqslant\rho(\mathcal{A}_{T^k})\leqslant\rho(\mathcal{A}_{S_n^k}),
\end{eqnarray*}
where the left equality holds if and only if $T=P_n$, and the right equality holds if and only if $T=S_n$.
\end{proof}

\begin{thm}\label{thm4.7}
If $\alpha\neq0$ is an eigenvalue of a $d$-regular graph $G$, then the roots of $(x-d)(x-1)^{\frac{k-2}{2}}-\alpha=0$ are signless Laplacian eigenvalues of $G^k$. Moreover, $\rho(\mathcal{Q}_{G^k})$ is the largest real root of $(x-d)(x-1)^{\frac{k-2}{2}}-d=0$.
\end{thm}
\begin{proof}
Suppose that $x$ is an eigenvector of the eigenvalue $\alpha\neq0$ of graph $G$. Then $\sum_{j\in N_G(i)}x_j=\alpha x_i$ for any $i\in V(G)$, where $N_G(i)$ is the set of all neighbors of $i$ in $G$. Let $\lambda\in\mathbb{C}$ be any number such that $(\lambda-d)(\lambda-1)^{\frac{k-2}{2}}=\alpha$, then $\lambda\neq1$. Let $y$ be a column vector of dimension $|V(G^k)|$ such that $y_u=(x_u)^{\frac{2}{k}}$ if $u\in V(G)$, and $y_u=(\lambda-1)^{-\frac{1}{2}}(x_ix_j)^{\frac{1}{k}}$ if $u\in V(G^k)\setminus V(G)$ is a core vertex in the edge contains two adjacent vertices $i,j\in V(G)$. For any $i\in V(G)$, by $\sum_{j\in N_G(i)}x_j=\alpha x_i$ and $(\lambda-d)(\lambda-1)^{\frac{k-2}{2}}=\alpha$, we have
\begin{eqnarray*}
(\mathcal{Q}_{G^k}y)_i=d(x_i)^{\frac{2(k-1)}{k}}+\sum_{j\in N_G(i)}(\lambda-1)^{-\frac{k-2}{2}}(x_ix_j)^{\frac{k-2}{k}}(x_j)^{\frac{2}{k}}=\lambda(x_i)^{\frac{2(k-1)}{k}}=\lambda(y_i)^{k-1}.
\end{eqnarray*}
For any $u\in V(G^k)\setminus V(G)$, we have
\begin{eqnarray*}
(\mathcal{Q}_{G^k}y)_u&=&(\lambda-1)^{-\frac{k-1}{2}}(x_ix_j)^{\frac{k-1}{k}}+(\lambda-1)^{-\frac{k-3}{2}}(x_ix_j)^{\frac{k-3}{k}}(x_i)^{\frac{2}{k}}(x_j)^{\frac{2}{k}}\\
&=&\lambda(\lambda-1)^{-\frac{k-1}{2}}(x_ix_j)^{\frac{k-1}{k}}=\lambda(y_u)^{k-1}.
\end{eqnarray*}
Hence $\lambda$ is a signless Laplacian eigenvalue of $G^k$ with an eigenvector $y$.

If $G$ is connected and $\alpha=d=\rho(\mathcal{A}_G)$, then we can choose $x$ as a positive eigenvector of $\rho(\mathcal{A}_G)$. In this case, $y$ is a positive eigenvector of the signless Laplacian eigenvalue $\lambda$ of $G^k$. Lemma \ref{lem4.3} implies that $\rho(\mathcal{Q}_{G^k})$ is the largest real root of $(x-d)(x-1)^{\frac{k-2}{2}}-d=0$ when $G$ is connected.

If $G$ has $r\geqslant2$ components $G_1,\ldots,G_r$, then
\begin{eqnarray*}
\rho(\mathcal{Q}_{G^k})=\max\{\rho(\mathcal{Q}_{G_1^k}),\ldots,\rho(\mathcal{Q}_{G_r^k})\}.
\end{eqnarray*}
Since $G_1,\ldots,G_r$ are connected $d$-regular graphs, we know that $\rho(\mathcal{Q}_{G^k})=\rho(\mathcal{Q}_{G_1^k})=\cdots=\rho(\mathcal{Q}_{G_r^k})$ is equal to the largest real root of $(x-d)(x-1)^{\frac{k-2}{2}}-d=0$.
\end{proof}
The following corollary follows from Theorem \ref{thm4.7}.
\begin{cor}
For any $d$-regular graph $G$, we have $\lim_{k\rightarrow\infty}\rho(\mathcal{Q}_{G^k})=d$. Moreover, $\rho(\mathcal{Q}_{G^k})$ is a strictly decreasing sequence if $d>1$.
\end{cor}

\noindent
\textbf{Remark.} In [7, Conjecture 4.1], Hu et al conjectured that $\rho(\mathcal{Q}_{G^k})$ is a strictly decreasing sequence for any graph $G$ and even $k$. By Corollary 4.9, this conjecture holds when $G$ is regular of degree $d>1$.

\vspace{3mm}
The proof of the following theorem is similar with that of Theorem \ref{thm4.7}. So we omit it.
\begin{thm}
If $\alpha\neq0$ is an eigenvalue of a $d$-regular graph $G$, then the roots of $(d-x)(1-x)^{\frac{k-2}{2}}-\alpha=0$ are Laplacian eigenvalues of $G^k$.
\end{thm}

\end{document}